\documentclass[11pt]{amsart}
\usepackage{amssymb, amsmath, amsthm, amsfonts}
\usepackage{epsfig}
\usepackage{latexsym,amssymb}
\usepackage{times}
\newtheorem{defi}{Definition}[section]
\newtheorem{theorem}[defi]{Theorem}
\newtheorem{lema}[defi]{Lemma}
\newtheorem{prop}[defi]{Proposition}
\newtheorem{cor}[defi]{Corollary}
\newtheorem{remark}[defi]{Remark}

\def\diag{\mathop\mathrm{diag}\nolimits}
\def\tr{\mathop\mathrm{tr}}

\def\mat{{\rm Sym}}

\def\abs#1{\left| #1\right|}

\def\s{\mathbb{S}}
\def\sfe{\mathbb{S}^{n-1}}
\def\R{\mathbb{R}}

\def\N{\mathbb{N}}

\def\K{\mathcal{K}}

\def\c{C}

\def\cf{\mathfrak{C}}

\def\H{\mathcal{H}}

\begin{document}
\title[Monotonicity and concavity of integral functionals]{Monotonicity and 
concavity of integral functionals \\
involving area measures of convex bodies}
\author[Andrea Colesanti, Daniel Hug and Eugenia Saor\'{\i}n G\'omez]{Andrea Colesanti,
Daniel Hug and Eugenia Saor\'{\i}n G\'omez}
\address{Dipartimento di Matematica e Informatica ``U. Dini",
Viale Morgagni 67/A, 50134 Firenze, Italy}
 \email{colesant@math.unifi.it}
\address{Karlsruhe Institute of Technology (KIT), Department of Mathematics,
D-76128 Karlsruhe, Germany} \email{daniel.hug@kit.edu}
\address{Fakult\"at f\"ur Mathematik, Otto-von-Guericke Universit\"at Magdeburg, \newline
Universit\"atsplatz 2, D-39106 Magdeburg, Germany}
\email{eugenia.saorin@ovgu.de}
\date{\today}
\begin{abstract} For a broad class of integral functionals defined on the space of $n$-dimensional convex bodies, 
we establish  necessary and sufficient conditions  for monotonicity, and necessary conditions for the validity
of a Brunn-Minkowski type inequality.  In particular, we prove that a Brunn-Minkowski type inequality implies monotonicity, and that
a general Brunn-Minkowski type inequality is equivalent to the functional being a mixed volume.
\end{abstract}
\maketitle

\noindent

\section{Introduction}\label{intro}

For a broad class of homogeneous functionals $\bf F$ defined on $\K^n$, the space  of {\em convex bodies}
(non-empty compact convex sets) in $\R^n$, a Brunn-Minkowski type inequality of the following form holds true,
\begin{equation}\label{BM-0}
{\bf F}((1-t)K+tL)^{1/\alpha}\ge(1-t){\bf F}(K)^{1/\alpha}+t
{\bf F}(L)^{1/\alpha}
\end{equation}
for all $K,L\in\K^n$ and  $t\in[0,1]$, 
where $(1-t)K+tL$ is a Minkowski combination of $K$ and $L$, and $\alpha$ is the degree of homogeneity of ${\bf F}$. 
In other words, condition \eqref{BM-0} states that $\bf F^{1/\alpha}$ is concave on $\K^n$. 
The archetype of these inequalities is the classical Brunn-Minkowski inequality, in which $\bf F$ is the
$n$-dimensional volume functional (Lebesgue measure) and $\alpha=n$. This inequality is one of the
cornerstones of convex geometry  and connects this subject to
many other areas of mathematics. The interested reader is referred to the survey paper \cite{Gardner} by Gardner. Other important examples come
from the realm of convex geometry itself (intrinsic volumes, mixed volumes and many others) or from analysis (e.g., eigenvalues of elliptic operators,
various notions of capacities); see for instance  \cite{Colesanti05} and \cite{Schneider}.

\medskip

In many remarkable cases, a functional ${\bf F}$ which satisfies a Brunn-Minkowski type inequality is accompanied by other significant properties like continuity,
additivity, and monotonicity with respect to set inclusion. One of the purposes of this paper is to investigate the interplay between a 
Brunn-Minkowski type inequality and monotonicity for some integral functionals involving {\em area measures} of convex bodies (see
Section \ref{sec-preliminaries} for definitions and references). 
For a continuous function $f$ defined on the unit sphere $\sfe$ of $\R^n$ and an integer $i\in\{1,\dots,n-1\}$, we define
\begin{equation}\label{intro.1}
K\,\mapsto\,{\bf F}(K):=\int_{\sfe} f(u) \, S_i(K;du),
\end{equation}
where $S_i(K;\cdot)$ denotes the $i$th {\it area measure} of $K$. 
By the properties of area measures (\cite[Section 5.1]{Schneider}) and the continuity of $f$, the functional $\bf
F$ is continuous with respect to the Hausdorff metric, translation invariant and  homogeneous of degree $i$.

In the particular case where $f$ is the {\em support function} of some fixed convex body $L$, the functional 
 $\bf F$ is in fact a {\it mixed volume} and has two additional
interesting properties. The first is monotonicity with respect to set inclusion, which means that for all $K,L\in\K^n$,
\begin{equation}\label{monotonicity}
K\subset L\,\Rightarrow\,{\bf F}(K)\le{\bf F}(L).
\end{equation}
Second, ${\bf F}$ satisfies a Brunn-Minkowski type inequality \eqref{BM-0} with $\alpha=i$, that is, 
\begin{equation}\label{BM}
{\bf F}((1-t)K+tL)^{1/i} \ge
(1-t){\bf F}(K)^{1/i}+t{\bf F}(L)^{1/i},
\end{equation}
for all 
$K,L\in\K^n$ and  $t\in[0,1]$. Since mixed volumes are non-negative, the $i$th root is well defined.

For general $f$, other than a support function, we cannot expect  $\bf F$ to satisfy 
either \eqref{monotonicity} or \eqref{BM}. Let us examine the case $i=n-1$.
In \cite{McMullen90}, McMullen  proved that, in this case, \eqref{monotonicity}
implies that $f$ is a support function. A corresponding result for the Brunn-Minkowski inequality
has been recently established in \cite{Colesanti-Hug-Saorin}. There it is shown that \eqref{BM} implies
that $f$ is a support function. Hence, for $i=n-1$, both \eqref{monotonicity} and
\eqref{BM} are  equivalent to the fact that $\bf F$ is a mixed volume,
and therefore they are equivalent to each other.

These equivalences are no longer true for $i<n-1$. For instance when
$i=1$, the functional $\bf F$ is linear with respect to the Minkowski addition,
and, in particular, it satisfies \eqref{BM} for every choice of
$f$. On the other hand, as we will see in Theorem 
\ref{thm-mon-1}, ${\bf F}$ is not monotonic for every $f$.

\medskip

In the first part of this paper we find a condition on $f$ which is equivalent to monotonicity of $\bf F$. 
We first present this condition in the
smooth case, that is, for $f\in C^2(\sfe)$. We need to introduce some notation. 
For $u\in\sfe$, we define the $(n-1)\times(n-1)$ matrix
$$
Q(f,u):=(f_{ij}(u)+f(u)\delta_{ij})_{i,j=1}^{n-1},
$$
where $f_{ij}$ are the second covariant derivatives of $f$ with respect to an orthonormal
frame on $\sfe$ and $\delta_{ij}$ are the usual Kronecker symbols. Hence, 
$Q(f,u)$ is the spherical Hessian matrix of $f$ at $u$ plus $f(u)$ times the identity
matrix (see Section \ref{sec-preliminaries} again for details). This is a symmetric matrix, and we will denote by
$\lambda_i(u)$, $i=1,\dots,n-1$, its eigenvalues. Note that
if $\bar f$ denotes the 1-homogeneous extension of $f$ to $\R^n$ and
$x\in\R^n\setminus\{0\}$, then the set of eigenvalues of the Euclidean Hessian
matrix of $\bar f$ at $x$, denoted by $D^2\bar f(x)$, is $\{\lambda_1(u),\dots,\lambda_{n-1}(u),0\}$,
where $u=x/\|x\|\in\sfe$. In particular,  the convexity of $\bar f$ is equivalent to the fact that
$Q(f,u)$ is positive semi-definite for every $u$ (see \cite[Appendix]{Colesanti-Hug-Saorin}).

To state our main results we need the following definition. 

\begin{defi} Let $f\in C^2(\sfe)$ and $i\in\{1,\dots,n-1\}$. We say that $f$ satisfies condition $({\bf M})_i$
if for every $u\in\sfe$ and $I\subset\{1,\ldots,n-1\}$ with $|I|=n-i$, we have
\begin{equation}\label{condition-M-i}
\sum_{i\in I} \lambda_i(u)\ge 0,
\end{equation}
where $|I|$ denotes the cardinality of $I$.
\end{defi}

In other words, for any choice of $(n-i)$  eigenvalues of $Q(f,u)$, their sum is non-negative.
Note that if $f$ satisfies $({\bf M})_i$, for some $i$, then it also satisfies
$({\bf M})_j$ for every $j\le i$. The following result asserts
that condition $({\bf M})_i$ is equivalent to monotonicity of $\bf F$.

\begin{theorem}\label{thm-mon-1} Let $f\in C^2(\sfe)$ and $i\in\{1,\dots,n-1\}$. Then
the functional $\bf F$ defined by  \eqref{intro.1} satisfies \eqref{monotonicity} if and only if
$f$ satisfies condition $({\bf M})_i$.
\end{theorem}

In the case $i=n-1$ condition $({\bf M})_i$ amounts to the fact that each eigenvalue must be non-negative, that is,
$Q(f,u)$ is positive semi-definite everywhere on $\sfe$, and then the 1-homogeneous extension $\bar f$ of $f$ is
convex. But this  in turn is equivalent to saying that $f$ is a support function. Hence we have an alternative proof
of the result of McMullen \cite{McMullen90}, at least in the smooth case, but our procedure extends to the
general case $f\in C(\sfe)$, as the Theorem \ref{thm-mon-2} shows.

In the other limiting case,  $i=1$, condition \eqref{condition-M-i} means that the trace of $Q(f,u)$ is non-negative
for every $u$; equivalently, 
$$
{\rm trace}(D^2\bar f (x))=\Delta\bar f(x)\ge 0\quad\text{for all } x\ne0,
$$
where $\Delta$ denotes the Euclidean Laplace operator, i.e., $\bar f$ is a subharmonic function in $\R^n\setminus\{0\}$.

In general, condition \eqref{condition-M-i} is related to the so-called {\em $r$-convexity of $f$} or, more precisely, of its
1-homogeneous extension. We recall that a function $g\in C^2(\Omega)$, where $\Omega$ is an open subset of
$\R^n$, is said to be $k$-convex, for some $k\in\{1,\dots n\}$, if for every $x\in\Omega$ and for $j=1,\ldots,k$ the  $j$th elementary
symmetric function of the eigenvalues of $D^2 g(x)$ is non-negative. In particular, it can be seen that
$n$-convexity is equivalent to the usual convexity. It is known (see for instance \cite{Salani-PhD}, Prop. 1.3.3) that
if $g$ is $k$-convex, then, for every $x\in\Omega$ and for every choice of $n-k+1$ distinct eigenvalues of $D^2g(x)$, their sum is
non-negative. Hence we have the following corollary.

\begin{cor} Let $i\in\{1,\dots,n-1\}$. Let $f\in C^2(\sfe)$, and let $\bar f$ be its 1-homogeneous extension. If $\bar f$ is
$i$-convex in $\R^n\setminus\{0\}$, then the functional defined by \eqref{intro.1} is monotonic.
\end{cor}

Theorem \ref{thm-mon-1} is complemented by the following statement concerning the case in which $f$ is just continuous.

\begin{theorem}\label{thm-mon-2} Let $f\in C(\sfe)$ and let $i\in\{1,\dots,n-1\}$. Then
the functional $\bf F$ defined by  \eqref{intro.1} satisfies \eqref{monotonicity}, i.e., it is monotonic w.r.t. set inclusion, if and only if
there exists a sequence $f_k\in C^2(\sfe)$, $k\in\N$, converging to $f$ uniformly on $\sfe$ and such that
$f_k$ satisfies condition $({\bf M})_i$ for every $k\in\N$.
\end{theorem}

\medskip

In Section \ref{sec-BM},
we consider the case in which $\bf F$ satisfies a Brunn-Minkowski type inequality and prove the following theorem.

\begin{theorem}\label{thm-BM} Let $i\in\{2,\dots,n-1\}$ and let $f\in C^2(\sfe)$ be such that the functional $\bf F$ defined by \eqref{intro.1} is non-negative and satisfies
the Brunn-Minkowski type inequality \eqref{BM}. Then $f$ satisfies condition $({\bf M})_i$.
\end{theorem}

Theorem \ref{thm-BM} provides a {\em necessary} condition on $f$ so that  $\bf F$  satisfies \eqref{BM}. However we do not know whether
this condition is sufficient as well, apart from the case $i=n-1$ in which the answer is affirmative, as proved in \cite{Colesanti-Hug-Saorin}.
Theorem \ref{thm-BM} has the following corollary.

\begin{cor}\label{cor. BM implies mon.} Let $i\in\{2,\dots,n-1\}$ and let $f\in C^2(\sfe)$ be such that the functional $\bf F$ defined by \eqref{intro.1} is non-negative and satisfies
the Brunn-Minkowski inequality \eqref{BM}. Then $\bf F$ is monotonic.
\end{cor}

In the case where $f$ is an even function in the sense that $f(-u)=f(u)$ for every $u\in\sfe$, and continuous, 
we have the following extension of Theorem \ref{thm-BM}  and Corollary \ref{cor. BM implies mon.}
(in the spirit of Theorem \ref{thm-mon-2}).

\begin{theorem}\label{thm-BM-sym} Let $i\in\{2,\dots,n-1\}$, let $f\in C(\sfe)$ be even, and let
${\bf F}$ be defined as in \eqref{intro.1}. If $\bf F$ is non-negative and satisfies inequality \eqref{BM}, then there exists a sequence of functions $f_k\in C^2(\sfe)$,
$k\in\N$, which converges uniformly to $f$ on $\sfe$ such that $f_k$ satisfies condition $({\bf M})_i$ for every $k\in\N$.
In particular, $\bf F$ is monotonic.
\end{theorem}

In Section \ref{sec-BM} we will see that the previous result also holds when the symmetry assumption on $f$ is replaced by
the existence of second weak derivatives in the sense of Sobolev spaces.

\medskip

Functionals defined by means of \eqref{intro.1} can be seen as examples of more general integral functionals. Given $K_1,\dots,K_{n-1}\in\K^n$, let
$S(K_1,\dots,K_{n-1};\cdot)$ be their mixed area measure (see Section \ref{sec-preliminaries} for precise definitions). If $f\in C(\sfe)$ and
$i\in\{1,\dots,n-1\}$, let the functional ${\bf F}:\K^n \longrightarrow \R$ be defined by
\begin{equation}\label{F mixed area measure}
{\bf F}(K)=\int_{S^{n-1}} f(u)\, S(K[i],K_1,\dots,K_{n-i-1};du).
\end{equation}
The functional in \eqref{intro.1} is recovered from \eqref{F mixed area measure} in the special case where $K_1,\dots,K_{n-i-1}$ coincide with the Euclidean unit ball. If $f$ is the support function of some
convex body $L$, then $\bf F$ equals the mixed volume
$$
V(L,K[i],K_1,\dots,K_{n-i-1}).
$$
In this case, ${\bf F}$ satisfies a Brunn-Minkowski type inequality for any choice of $L,K_1,\dots,K_{n-i-1}\in\K^n$. This result is called 
{\em general Brunn-Minkowski inequality} (see \cite[Theorem 6.4.3]{Schneider}). In the last section of this paper we
prove that this property characterizes support functions.

\begin{theorem}\label{thm-general-BM}
Let $f\in C(S^{n-1})$ and $i\in\{2,\dots,n-1\}$  be such that for
any choice of convex bodies $K_1,\dots,K_{n-i-1}\in\K^n$ the functional ${\bf F}:\K^n
\longrightarrow \R$ defined by \eqref{F mixed area measure} is non-negative and satisfies
\begin{equation}\label{WBM}
{\bf F}((1-t)K+tL)^{1/i}\geq
(1-t){\bf F}(K)^{1/i}+t\,{\bf F}(L)^{1/i}
\end{equation}
for all $ t\in[0,1]$ and $ K,L\in\K^n$. 
Then $f$ is the support function of a convex body.
\end{theorem}

The  general Brunn-Minkowski inequality \eqref{WBM} for the functional
$\bf F$ defined as in \eqref{F mixed area measure} implies that
\begin{equation}\label{BM-min}
{\bf F}((1-t)K+tL)\ge\min\{{\bf F}(K),{\bf F}(L)\}
\end{equation}
for all $K,L\in\K^n$ and $ t\in[0,1]$, which is in general weaker than \eqref{WBM}. 
However, in many cases it can be shown to be equivalent  to it 
 by a standard argument based  on homogeneity. Note that \eqref{BM-min} does not require 
$\bf F$ to be non-negative {\em a priori}. 

The characterization theorem proved in \cite{Colesanti-Hug-Saorin} for the functional $\bf F$ defined by \eqref{intro.1} in the case
$i=n-1$ was proved under the assumption \eqref{BM-min}. This leads to the following extension of 
Theorem \ref{thm-general-BM} in which condition \eqref{WBM} of Theorem \ref{thm-general-BM} is replaced
by \eqref{BM-min}  and the requirement that ${\bf F} $ be non-negative is removed. 

\begin{theorem}\label{thm-general-BM+}
Let $f\in C(S^{n-1})$ and $i\in\{2,\dots,n-1\}$  be such that for
any choice of convex bodies $K_1,\dots,K_{n-i-1}\in\K^n$ the functional ${\bf F}:\K^n
\longrightarrow \R$ defined by \eqref{F mixed area measure} satisfies \eqref{BM-min}. 
Then $f$ is the support function of a convex body.
\end{theorem}

\section{Preliminaries}\label{sec-preliminaries}

We work in the $n$-dimensional Euclidean space $\R^n$, $n\ge2$, endowed with
the usual scalar product $\langle\cdot,\cdot\rangle$ and norm $\|\cdot\|$. We denote by $B^n$ the
closed unit ball centered at the origin, and by $\sfe$ the unit sphere.
Throughout the paper we will often use the  convention that we sum over
repeated indices.

\subsection{Convex bodies}
As stated in the introduction, for $n\ge1$ we denote by $\K^n$ the collection
of all non-empty compact convex subsets of $\R^n$, which are called {\em convex bodies}, for short.
Our reference text on the theory of convex bodies is the monograph
\cite{Schneider} by Schneider. Given $K,L\in \K^n$ and
$\alpha,\beta\ge0$, we write 
$\alpha K+\beta L=\{\alpha x+\beta y\ |\,x\in K,\ y\in L\}$ 
for the {\em Minkowski combination} of $K$ and $L$ with coefficients $\alpha$ and $\beta$. 

For $K\in\K^n$ we denote by $h_K$ the {\it support function} of
$K$, considered as a function on the unit sphere. 
We recall that support functions behave linearly with respect
to the operations introduced above. For $K,L\in \K^n$ and
$\alpha,\beta\ge0$, we have $h_{\alpha K+\beta L}=\alpha h_K+\beta h_L$. 
Another property of convex bodies which can be expressed in a simple way
via support functions is set inclusion. Indeed, for  $K,L\in\K^n$,
\begin{equation}\label{I.1b}
K\subset L\quad\mbox{if and only if}\quad h_K\le h_L \text{ on $\sfe$.}
\end{equation}

\medskip


We will frequently need to work with convex bodies whose boundary is smooth. Let us
introduce the following notation. 
We say that a convex body $K$ with non-empty interior is of class
$C^2_+$ (briefly, $K\in C^2_+$), if its boundary is of class $C^2$
and the Gauss curvature is strictly positive at every boundary point $x\in\partial K$.

For $\phi\in C^2(\sfe)$, $u\in\sfe$, and $i, j \in\{ 1,\dots,n-1\}$, we
put
\begin{equation*}
q_{ij}(\phi, u) := \phi_{ij} (u)+ \delta_{ij}\phi(u),
\end{equation*}
where $\phi_{ij}$ denote the second covariant derivatives of
$\phi$, computed with respect to a local orthonormal frame (of
vector fields) on $\sfe$ and $\delta_{ij}$ denote the usual Kronecker
symbols. Moreover we set
\begin{equation}\label{I.2}
Q(\phi,u)=\left( q_{ij}(\phi,u) \right)_{i,j=1}^{n-1}.
\end{equation} 
All relevant quantities and conditions will be independent of the
particular choice of a local orthonormal frame in the following.
For the sake of brevity, we sometimes omit the variable $u$ and
simply write $q_{ij}(\phi)$ or $Q(\phi)$. Note that the matrix
$Q(\phi,u)$ is {\it symmetric} for every $\phi\in C^2(\sfe)$ and
every $u\in\sfe$ (see \cite[Section 2]{Colesanti-Hug-Saorin} for
further  details). In the special case when $\phi$ is a support
function, the matrix $Q(\phi,\cdot)$ will play a crucial role in the sequel. 

We set
$$
\cf=\{h\in C^2(\sfe)\ |\ Q(h,u)>0\ \text{ for all } u\in\sfe\}\, ,
$$
where the notation $A>0$ stands for the matrix $A$ being positive
definite.

A proof of the following result can be deduced from \cite[Sect.
2.5]{Schneider}.
\begin{prop}\label{propI.1}
If $K\in C^2_+$, then $h_K\in\cf$. Conversely, if $h\in\cf$, then there
exists a uniquely determined $K\in C^2_+$ such that $h=h_K$.
\end{prop}

\medskip


The {\em mixed volume} of the convex bodies $K_1,\dots,K_n\in\K^n$ is denoted by 
$V(K_1,\dots,K_n)$.  For the {\em mixed area measure} of $K_1,\dots, K_{n-1}\in\K^n$, we write  $S(K_1,\dots,K_{n-1};\cdot)$;  
see \cite[Chapter 5]{Schneider} for the definitions. If in one of these functionals a convex body $K$ is repeated $i$ times,  
we use the notation $K[i]$, for instance,   we put 
$$
V(K[i],K_{i+1},\dots,K_{n}):=
V(\underbrace{K,\dots, K}_{\mbox{\tiny $i$-times}},K_{i+1},\dots,K_n).
$$
%
%
The mixed are measures are Borel measures defined on  $\sfe$. For the properties of area measures we refer
to \cite[Section 5.1]{Schneider}. The close connection between mixed volumes and mixed area measures is expressed by the relation
$$
V(K_1,\dots,K_n)=\frac1n\int_{\sfe}h_{K_n}(u)\, S(K_1,\dots,K_{n-1};du).
$$
For a given a convex body $K$ and $i\in\{1,\dots,n-1\}$, the {\em $i$th area measure} of $K$ 
is denoted by $S_i(K,\cdot)$ and equals the special mixed area measure 
 $S(K[i],B^n[n-i-1];\cdot)$.

\medskip

For the proof of our main results it will be important to
express the density of the area measures of a convex body $K$
in terms of the matrix $Q(h_K)$. Before stating such representations 
we need to recall some facts about elementary symmetric functions.

\subsection{Elementary symmetric functions and densities of area measures}\label{subsec-esm}
Let $N$ be an integer. We denote by $\mat(N)$ the set of $N\times
N$ symmetric matrices (with real entries). For an element
$A\in\mat(N)$ we write $A>0$ and $A\ge0$ if $A$ is positive
definite and positive semi-definite, respectively.

Let $A=(a_{jk})^N_{j,k=1}\in\mat(N)$, with eigenvalues
$\lambda_j$, $j=1,\dots,N$, and let $i\in\{0,1,\dots,N\}$. We define
$S_i(A)$ as the $i$th {\it elementary symmetric function}  of the
eigenvalues of $A$, that is,
$$
S_i(A)=\sum_{1\le j_1<\dots<j_i\le N}
\lambda_{j_1}\cdots\lambda_{j_i}\quad\mbox{if $i\ge1$,}
$$
and $S_0(A)=1$. Note, in particular, that $S_1(A)$ and $S_N(A)$ are
the trace and the determinant of $A$, respectively. An explicit description 
of $S_i(A)$ in terms of (the entries of) $A$ is provided in \eqref{SiA} below. 


For $N$, $A$
and $i$ as above, and for $j,k\in\{1,\dots,N\}$, we set
$$
S_i^{jk}(A)=\frac{\partial S_i}{\partial a_{jk}}(A).
$$
The $N\times N$ matrix consisting of the entries $S_i^{jk}(A)$
is sometimes called the {\em $i$th cofactor matrix of $A$}.
We will also need the second derivatives of $S_i(A)$ with respect
to the entries of $A$, which are denoted by
$$
S_i^{jk,rs}(A):=\frac{\partial^2 S_i}{\partial a_{jk}\partial
a_{rs}}(A),
$$
for every $i,j,r,s\in\{1,\dots,N\}$.

\bigskip

Let $K$ be a convex
body of class $C^2_+$ and $h\in\cf$  its support function. For  
 $i\in\{1,\dots,n-1\}$, the $i$th area measure $S_i(K;\cdot)$ of $K$ 
is absolutely continuous with respect to the Haussdorf measure $\H^{n-1}$
restricted to $\sfe$, and its density is given by the function
$$
u\mapsto S_i(Q(h),u),\qquad u\in \mathbb{S}^{n-1},
$$
(see, for example, \cite[5.3.2]{Schneider} for a proof).
In other words, for every $f\in C(\sfe)$ we have
$$
{\bf F}(K)=\int_{\sfe} f(u)S_i(Q(h,u))\, \H^{n-1}(du).
$$

\subsection{A lemma of Cheng and Yau}

For $\phi\in C^2(\sfe)$ and $i\in\{1,\dots,n-1\}$, consider the matrix
\begin{equation}\label{I.10}
(S_i^{jk}(Q(\phi,u)))_{j,k=1}^{n-1}
\end{equation}
as a function of $u\in\sfe$.
The following lemma will be of great
importance in the rest of this paper. It asserts that if we
consider any of the columns of \eqref{I.10} as a vector field on
$\sfe$, its divergence vanishes pointwise. 
The case $k=n-1$ was originally proved by Cheng and Yau
in \cite{Cheng-Yau}.

\begin{lema}\label{Cheng-Yau} Let $\phi\in C^3(\sfe)$ and $i\in\{1,\dots,n-1\}$. Then, for every $k\in\{1,\dots,n-1\}$, 
$$
\sum_{j=1}^{n-1}\left( S_i^{jk}(Q(\phi,u))
\right)_j=0\quad\text{for all }\ u\in\sfe.
$$
\end{lema}

We will also need a further generalization of Lemma \ref{Cheng-Yau}. Let $\phi, \psi\in C^3(\sfe)$. Then, for $u\in\sfe$, we 
define the matrix $M=M(u)=(m_{jk}(u))_{j,k=1,\dots,n-1}$ by
$$
m_{jk}(u)=S_i^{jk,rs}(Q(\phi,u))q_{rs}(\psi, u)
$$
(remember that we use the summation convention).
\begin{lema}\label{Cheng-Yau-ext}
In the above notation, for every $k\in\{1,\dots,n-1\}$,
\begin{equation*}
\sum_{j=1}^{n-1}(m_{jk}(u))_j=0\quad\forall\,u\in\sfe\,.
\end{equation*}
\end{lema}
\begin{proof} The proof follows the  argument used in the proof of \cite[Lemma
  1]{Colesanti-Saorin}. We use an explicit formula
for the $i$th cofactor matrix in terms of the entries of the original
matrix (see for instance \cite{Reilly} or \cite{Salani-PhD}). For $A\in\mat(n-1)$ we have
\begin{equation}\label{SiA}
S_i(A)=\frac{1}{i!}\sum{\delta\binom{j_1,\dots,j_{i}}{k_1,\dots,k_{i}}a_{j_1
k_1}\cdots a_{j_{i} k_{i}}},
\end{equation}
where the sum is taken over all possible indices $j_s, k_s \in
\{1,\dots,n-1\}$ (for $s=1,\dots,i$) and the Kronecker symbol
$\delta\binom{j_1,\dots,j_{i}}{k_1,\dots,k_{i}}$ equals $1$
(respectively, $-1$) when $j_1,\dots,j_{i}$ are distinct and
$(k_1,\dots,k_{i})$ is an even (respectively, odd) permutation
of $(j_1,\dots,j_{i})$; otherwise it is $0$. Using the above
equality, we have, for every $j,k,r,s\in\{1,\dots,n-1\}$,
\begin{eqnarray}
S^{jk}_i(A)&=&\frac{1}{(i-1)!}\sum{\delta\binom{j,j_1,\dots,j_{i-1}}{k,j_1,\dots,k_{i-1}}a_{j_1
k_1}\cdots a_{j_{i-1} k_{i-1}}}\,,\nonumber\\
\nonumber\\
S^{jk,rs}_i(A)&=&\frac{1}{(i-2)!}\sum{\delta\binom{r,j,j_1,\dots,j_{i-2}}{s,k,k_1,\dots,k_{i-2}}a_{j_1
k_1}\cdots a_{j_{i-2} k_{i-2}}}\,.
\end{eqnarray}

For simplicity, in the following formulas we omit the variable
$u\in\sfe$. Then for the matrix $m_{jk}$ we obtain
\begin{eqnarray*}
m_{jk}=\frac{1}{(i-2)!}\sum_{r,s}\sum{\delta\binom{r,j,j_1,\dots,j_{i-2}}{s,k,k_1,\dots,k_{i-2}}
q_{j_1k_1}(\phi)\cdots q_{j_{i-2}k_{i-2}}}(\phi)\,q_{rs}(\psi)\,.
\end{eqnarray*}
Hence
\begin{eqnarray*}
&&(i-2)!\sum_{j=1}^{n-1} (m_{jk})_j=\\
&&=\sum_{j=1}^{n-1}\sum_{r,s}
\sum \Bigg\{\delta\binom{r,j,j_1,\dots,j_{i-2}}{s,k,k_1,\dots,k_{\,i-2}}\times\nonumber\\
&&\times\Bigl[(\phi_{j_1 k_1j}+\phi_j\delta_{j_1k_1})(\phi_{j_2 k_2}+\phi\delta_{i_2j_2})\cdots
    (\phi_{j_{i-2}k_{i-2}}+\phi\delta_{j_{i-2}k_{i-2}})+\cdots\nonumber\\
&&+(\phi_{j_1k_1}+\phi\delta_{j_1k_1})\cdots(\phi_{j_{i-3} k_{i-3}}+\phi\delta_{j_{i-3}k_{i-3}})
    (\phi_{j_{i-2}k_{i-2}j}+\phi_j\delta_{j_{i-2}k_{i-2}})\Bigr](\psi_{rs}+\delta_{rs}\psi)+\nonumber\\
&&+(\phi_{j_1k_1}+\phi\delta_{j_1k_1})\cdots(\phi_{j_{i-2}j_{i-2}}
+\phi\delta_{j_{i-2}k_{i-2}})(\psi_{rsj}+\delta_{rs}\psi_j)\Bigg\}.\nonumber
\end{eqnarray*}
In the last sum, for fixed
$j_1,\dots,j_{i-2},k_1,\dots,k_{i-2},j,r,s$, we split the terms into two types:
those in which there are no third
covariant derivatives of $\psi$, and those where a 
third derivative of $\psi$ appears. As for the first type, consider the terms
\[
A=\delta_1(\phi_{j_1k_1j}+\phi_j\delta_{j_1k_1})C\quad\text{ and }\quad
B=\delta_2(\phi_{jk_1j_1}+\phi_{j_1}\delta_{jk_1})C\,,\;
\]
where
\[
\delta_1=\delta\binom{r,j,j_1,j_2,\dots,j_{i-2}}{s,k,k_1,k_2,\dots,k_{i-2}},\quad
\delta_2=\delta\binom{r,j_1,j,j_2,\dots,j_{i-2}}{s,k,k_1,k_2,\dots,k_{i-2}},
\]
and
\[
C=(\phi_{j_2k_2}+\phi\delta_{j_2k_2})\cdots(\phi_{j_{i-2}k_{i-2}}+\phi\delta_{j_{i-2}k_{i-2}})(\psi_{rs}+\delta_{rs}\psi).
\]
Clearly $\delta_2=-\delta_1$. On the other hand,  the 
 third order covariant derivatives of a function $g\in\c^3(\sfe)$ satisfy the symmetry relations
$$
g_{\alpha\beta\gamma}=g_{\beta\alpha\gamma}\,,\quad\alpha\,,\,\beta\,,\,\gamma=1,\dots,n-1\,,
$$
and
$$
g_{\alpha\beta\gamma}+g_\gamma\delta_{\alpha\beta}\equiv
g_{\alpha\gamma\beta}+g_\beta\delta_{\alpha\gamma}\,,\quad\alpha\,,\,\beta\,,\,\gamma=1,\dots,n-1\,.
$$
Consequently, 
\begin{align*}
A+B&=\delta_1C\bigl(\phi_{j_1k_1j}+\phi_j\delta_{j_1k_1}-\phi_{jk_1j_1}-\phi_{j_1}\delta_{jk_1}\bigr)\\
&=\delta_1C\bigl(\phi_{k_1j_1j}+\phi_j\delta_{j_1k_1}-\phi_{jk_1j_1}-\phi_{j_1}\delta_{jk_1}\bigr)\\
&=\delta_1C\bigl(\phi_{k_1jj_1}+\phi_{j_1}\delta_{jk_1}-\phi_{jk_1j_1}-\phi_{j_1}\delta_{jk_1}\bigr)\\
&=0.
\end{align*}
For any term $A$ (of the mentioned type) in
the above sum, there exists another term $B$, uniquely determined,
which cancels out with $A$. 

Concerning the terms of the second type, consider
the summands
\[
E=\delta_3(\psi_{rsj}+\psi_j\delta_{rs})D\quad\text{ and }\quad
F=\delta_4(\psi_{jsr}+\psi_{r}\delta_{js})D\,,\;
\]
where
\[
\delta_3=\delta\binom{r,j,j_1,j_2,\dots,j_{i-2}}{s,k,k_1,k_2,\dots,k_{i-2}},\quad
\delta_4=\delta\binom{j,r,j_1,j_2,\dots,j_{i-2}}{s,k,k_1,k_2,\dots,k_{i-2}},
\]
and
\[
D=(\phi_{j_1k_1}+\phi\delta_{j_1k_1})\cdots(\phi_{j_{i-2}k_{i-2}}+\phi\delta_{j_{i-2}k_{i-2}})\,.
\]
Again, it is clear that $\delta_3=-\delta_4$, and by the same reasoning as before we get $E+F=0$, 
which concludes the proof.
\end{proof}

\begin{remark}\label{rem-int-by-parts}{\rm 
As a consequence of Lemma \ref{Cheng-Yau}, together with the divergence
theorem applied twice on the sphere and the definition
\eqref{I.2} of the matrix $Q$, it is easy to prove that, for $h\in
\cf$ and $f,\phi\in C^2(\sfe)$, 
\[
\begin{split}\label{I.15}
\int_{\sfe}f S_i^{kj}(Q(h)) q_{kj}(\phi )\,d\H^{n-1}
&=\int_{\sfe}\left(f\phi\,{\rm trace}(S_i^{kj}(Q(h)))+fS_i^{kj}(Q(h))\phi_{kj}\right)\,d\H^{n-1}\\
&=\int_{\sfe}\left(f\phi\,{\rm trace}(S_i^{kj}(Q(h)))-S_i^{kj}(Q(h))f_j\phi_{k}\right)\,d\H^{n-1}\\
&=\int_{\sfe}\left(f\phi\,{\rm trace}(S_i^{kj}(Q(h)))+\phi S_i^{kj}(Q(h))f_{kj}\right)\,d\H^{n-1}\\
&=\int_{\sfe}\phi S_i^{kj}(Q(h)) q_{kj}(f)\,d\H^{n-1}.
\end{split}
\]
By Lemma \ref{Cheng-Yau-ext}, the same conclusion holds if we
replace the matrix $\left(S_i^{jk}(Q(h))\right)_{j,k=1,\dots,n-1}$
by the matrix $\left(S_i^{jk,rs}(Q(h))q_{rs}(\phi)\right)_{j,k=1\dots,n-1}$. Note that here we assume that 
$\phi\in C^2(\sfe)$, while Lemma \ref{Cheng-Yau} and Lemma \ref{Cheng-Yau-ext} are stated for functions of class $C^3$. 
The extension follows by a straightforward approximation argument.}
\end{remark}

\subsection{Mollification}
We recall a standard method to approximate continuous functions on
the unit sphere by smooth functions. Let $\psi:\R\to[0,\infty)$ be
a function of class $C^\infty$ with ${\rm
sprt}(\psi)\subset [-1,1]$ and $\psi(0)>0$. Then, for $k\in\N$,
we define $\omega_k:\mathbf{O}(n)\to [0,\infty)$ by $
\omega_k(\rho):=c_k\cdot \psi(k^2\cdot \|\rho-{\rm id}\|^2)\,, $
where $\mathbf{O}(n)$ is the group of rotations of $\R^n$ endowed
with the Haar probability measure $\nu$, ``${\rm id}$'' is the
identity element in ${\mathbf{O}(n)}$ and $c_k$ is chosen such
that $ \int_{\mathbf{O}(n)}\omega_k(\rho)\, \nu(d\rho)=1. $ As a
composition of $C^\infty$ maps, $\omega_k$ is of class $C^\infty$.
The following lemma is standard.
\begin{lema}
\label{l:regularization} Let $f\in\c(\s^{n-1})$. Then, for
$k\in\N$, the function $f_k:\s^{n-1}\to\R$ defined by
\begin{equation}\label{f_k mollification}
f_k(u):=\int_{\mathbf{O}(n)}f(\rho u) \omega_k(\rho)\, \nu(d\rho)\,,\qquad u\in \s^{n-1}\,,
\end{equation}
is of class $C^\infty (\sfe)$, and the sequence $(f_k)_{k\in\N}$
converges to $f$ uniformly on $\s^{n-1}$.
\end{lema}

\section{Conditions for monotonicity}\label{sec-monotonicity}

In this section we prove Theorems \ref{thm-mon-1} and \ref{thm-mon-2}. We recall
that $F$ is said to be ``monotonic'', when  $F$ 
is increasing with respect to set inclusion (see \eqref{monotonicity}).

Let $K\in\K^n$ be of class $C^2_+$ and let $h$ be its support
function, hence $h\in{\mathfrak C}$. If $\phi\in C^2(\sfe)$, then there
exists $\epsilon>0$ such that
\[
h_s:=h+s\phi\in\cf\quad\mbox{for every $s$ such that
$|s|\le\epsilon$.}
\]
Hence, for every $s\in[-\epsilon,\epsilon]$ there exists a convex
body $K_s$ of class $C^2_+$ such that $h_s=h_{K_s}$ (by Proposition
\ref{propI.1}). Note that \eqref{I.1b} implies that $\phi\ge 0$ if and only if 
$K_{s_1}\subset K_{s_2}$ whenever $-\epsilon\le s_1\le s_2\le\epsilon$. 

The quantity ${\bf F}(K_s)$ is well defined for $|s|\le\epsilon$, 
and  its derivative at $s=0$ is given by
\begin{equation}\label{II.5}
\left.\frac d{ds} {\bf F}(K_s) \right|_{s=0}= \int_{\sfe}f
S_i^{kj}(Q(h)) q_{kj}(\phi)\,d\H^{n-1}.
\end{equation}

Next assume that $\bf F$ is monotonic and let $\phi$ be
non-negative on $\sfe$. Then  $s\mapsto {\bf F}(K_s)$ is an
increasing function for $|s|\le\epsilon$ so that
\begin{equation}\label{II.10}
\int_{\sfe}f S_i^{kj}(Q(h)) q_{kj}(\phi)\,d\H^{n-1}\ge0.
\end{equation}

Conversely, assume  that \eqref{II.10} holds for every $h\in\cf$ and
every non-negative $\phi\in C^2(\sfe)$. Let $K$ and $L$ be convex
bodies of class $C^2_+$ such that $K\subset L$ and define
$$
H(s)={\bf F}((1-s)K+sL),\quad s\in[0,1].
$$
As above we get
$$
H'(s)= \int_{\sfe}f S_i^{kj}(Q(h_s))
q_{kj}(h_L-h_K)\,d\H^{n-1},
$$
where $h_s=(1-s)h_K+sh_L$. Since $K\subset L$, we have $h_L-h_K\ge0$ on
$\sfe$. If we apply \eqref{II.10} with $\phi=h_L-h_K$, we obtain
that $H$ is increasing. Hence ${\bf F}(K)=H(0)\le H(1)={\bf
F}(L)$. This means that $\bf F$ is monotonic if restricted to
 convex bodies of class $C^2_+$; but as  convex bodies of class $C^2_+$ are dense in
$\K^n$ and $\bf F$ is continuous, we deduce that $\bf F$ is
monotonic on $\K^n$. Thus we have proved the following statement.

\begin{prop}\label{propII.1}
Assume that $i\in\{1,\dots,n-1\}$ and $\bf F$ is given by \eqref{intro.1} with $f\in
C(\sfe)$. Then $\bf F$ is monotonic on $\K^n$ if and only if
\begin{equation}\label{II.15}
\int_{\sfe}f S_i^{kj}(Q(h)) q_{kj}(\phi)\,d\H^{n-1}\ge0
\end{equation}
for all $ h\in\cf$ and all $\phi\in
C^2(\sfe)$ with $\phi\ge0$ on $\sfe$.
\end{prop}

\bigskip

Assume that $f\in C^2(\sfe)$ and that $\bf F$ is monotonic. Then Remark \ref{rem-int-by-parts} 
implies that in \eqref{II.15} 
the roles of $f$ and $\phi$  can be interchanged so that
$$
\int_{\sfe}\phi S_i^{kj}(Q(h)) q_{kj}(f)\, d\H^{n-1}\ge0
$$
for all $ h\in\cf$ and all $\phi\in
C^2(\sfe)$ with $\phi\ge0$ on $\sfe$. 
From this we infer  
the pointwise condition
\begin{equation}\label{II.20}
S_i^{kj}(Q(h,u)) q_{kj}(f,u)\ge0 
\end{equation}
for all $h\in\cf$ and $ u\in\sfe$. 

The converse is obviously true as well, that is, \eqref{II.20} implies the
integral condition \eqref{II.15}.

\begin{prop}\label{propII.2} Assume that $i\in\{1,\dots,n-1\}$ and $\bf F$ is 
given by \eqref{intro.1} with $f\in C^2(\sfe)$.
Then $\bf F$ is monotonic on $\K^n$ if and only if \eqref{II.20} holds.
\end{prop}

In order to further investigate condition \eqref{II.20}, we
need the following result.

\begin{lema}\label{lemmaII.1} Let $A\in\mat(n-1)$, $A>0$, and let $u\in\sfe$. Then there exists a (symmetric) convex body
$K\in\K^n$ of class $C^2_+$ such that
$$
Q(h_K,u)=A.
$$
\end{lema}

\begin{proof} We first consider the case $u=(0,\dots,0,1)$ and $A=\diag\{A_1,\dots,A_{n-1}\}$, $A_k>0$ for every
$k=1,\dots,n-1$. We set $A_n=1$. The function $\bar h\,:\,\R^n\to\R$ defined by
\[
\bar h(x)=h(x_1, \ldots,
x_n)=\left(\sum_{k=1}^{n}A_k\,x_k^2\right)^{1/2}
\]
is convex, and it is the 1-homogeneous extension of the support function $h=h_{\mathcal E}$ of an ellipsoid
$\mathcal E$.  
For  $x\ne0$ we have
\[
\frac{\partial \bar h}{\partial
x_i}(x)=\frac{A_i\,x_i}{\bar h(x)}
\]
and
\[
\frac{\partial^2
\bar h}{\partial\,x_i\partial\,x_j}(x)=\frac{A_i\delta_{ij}}{\bar h(x)}-\frac{A_i
A_j x_i x_j}{\bar h^3(x)}.
\]
The (Euclidean) Hessian matrix of $\bar h$ at $u$ is
\[
D^2\bar h(u)=\begin{pmatrix}
A_1 &0  &\dots &0 &0\\
0 &A_2 & \dots &0 &0\\
\vdots &\vdots &\ddots &0&0\\
0 & \dots &0 &A_{n-1}&0\\
0 &0 & \cdots &0&0\\ 
\end{pmatrix}.
\]
To compute the covariant derivatives of $h$, we can use the usual
partial derivatives of $\bar h$ (see \cite[\S2.5]{Schneider}
and also \cite[Appendix A.2]{Colesanti-Hug-Saorin}) to obtain that 
\[
Q(h,u)=\diag\{A_1,\dots,A_{n-1}\},
\]
which finishes the proof in the case where $A>0$ is diagonal and $u=(0,\dots,0,1)$.
In the general case, let $T$
be an orthogonal $(n-1)\times(n-1)$ matrix such that
$T A T^t=\diag\{A_1,\dots,A_{n-1}\}$.
Choose a coordinate system such that $u=T((0,\dots,0,1))$, and repeat the above
construction of the ellipsoid $\mathcal E$ for the matrix $TAT^t$ with respect to such a system.
Then we have that $T A T^t= Q(h,(0,\dots,0,1))$.
Using again the Euclidean derivatives to calculate the covariant derivatives (see \cite[\S2.5]{Schneider}
and also \cite[Appendix A.2]{Colesanti-Hug-Saorin}),
it is not difficult to see that $A=T^t Q(h,(0,\dots,0,1)) T=
Q(h,T(0,\dots,0,1))=Q(h,u)$, which concludes the proof.
\end{proof}

By Proposition \ref{propII.1} and the above lemma, we immediately
obtain the following result.

\begin{prop}\label{propII.3} Assume that $i\in\{1,\dots,n-1\}$ and
$\bf F$ is given by \eqref{intro.1} with $f\in C^2(\sfe)$. Then $\bf F$ is monotonic in
$\K^n$ if and only if
\begin{equation}\label{II.25}
S_i^{kj}(A) q_{kj}(f,u)=
\tr((S_i^{kj}(A))\cdot Q(f,u))\ge0 
\end{equation}
for all $A\in\mat(n-1)$, $A>0$, and 
$u\in\sfe$.
\end{prop}

Next we further study condition
\eqref{II.25}. Let $N=n-1$. Given the
matrix $B=\diag\{b_1,\dots,b_N\}$, we write $\diag\{\hat{b_j}\}$
to denote the $(N-1)\times(N-1)$ matrix $\diag\{b_1,\dots
b_{j-1},b_{j+1},\dots,b_N\}$ obtained from $B$ by removing $b_j$
from the diagonal. We notice, that if $A\in \mat(N)$ has the eigenvalues $\lambda_1,\dots,\lambda_N$, 
then the matrix $(S_i^{kj}(A))$ has the eigenvalues 
$\frac{\partial S_i(A)}{\partial \lambda_\ell}=
S_{i-1}(\diag\{\hat{\lambda_\ell}\})$, $\ell=1,\ldots,N$ (see \cite[Proposition 1.4.1]{Salani-PhD}). 
For a fixed $u\in\sfe$, we denote by $M$ the matrix
$Q(f,u)\in \mat(N)$. By a proper choice of the coordinate system, we may assume that $M$ is diagonal, 
$M=\diag\{\mu_1,\dots,\mu_N\}$, and that $ (S_i^{kj}(A))$ is diagonal as well.  
Therefore we can restate 
condition \eqref{II.25} in the form
$$
\sum_{j=1}^N \mu_j S_{i-1}(\diag\{\hat{\lambda_j}\})=\tr((S_i^{kj}(A)) M)
\ge0 
$$
for every $A=\diag{\{\lambda_1,\dots,\lambda_N\}}>0$.

By a standard continuity argument  the latter is equivalent to
\begin{equation}\label{II.25 reform}
\sum_{j=1}^N \mu_j S_{i-1}(\diag\{\hat{\lambda_j}\})=\tr((S_i^{kj}(A)) M)
\ge0 
\end{equation}
for every $A=\diag{\{\lambda_1,\dots,\lambda_N\}}\geq 0$.

Using its equivalent form \eqref{II.25 reform},
we will prove that \eqref{II.25} for the matrix
$M$ is equivalent to condition (${\bf M}_i$) expressed by
\eqref{condition-M-i} in the introduction.

\begin{lema}\label{lemma-equiv}
Let $i\in\{1,\dots,N\}$ and let $M\in\mat(N)$. Then condition \eqref{II.25 reform} holds if
and only if
\begin{equation}\label{II.25 equiv}
\sum_{j\in J}\mu_{j}\geq 0
\end{equation}
for all  $J\subset\{1,\ldots,N\}$ of cardinality $|J|=N-i+1$.
\end{lema}

\begin{proof}
It is straightforward to prove that condition \eqref{II.25 reform} implies
\eqref{II.25 equiv} by evaluating the inequality for positive
semidefinite matrices $A=\diag\{\lambda_1,\dots,\lambda_N\}$ with
$\lambda_k\in \{0,1 \}$ for $k=1,\dots,N$. Indeed, it is enough to
consider all such matrices where $N-i+1$ elements in the diagonal
vanish and the remaining $i-1$ entries are equal to one.

For the converse, observe that
\begin{align*}
\sum_{j=1}^N \mu_j S_{i-1}(\diag{\{\hat{\lambda_j}\}})
&=\sum_{j=1}^N \mu_j \sum_{|I|=i-1}\left[\left(\prod_{i\in I}\lambda_i\right)\mathbf{1}_{\{j\notin I\}}\right]\\
&=\sum_{|I|=i-1}\sum_{j=1}^N \mu_j \mathbf{1}_{\{j\notin I\}}\left(\prod_{i\in I}\lambda_i\right)\\
&=\sum_{|I|=i-1} \left(\prod_{i\in I}\lambda_i\right)\sum_{j=1}^N\mu_j \mathbf{1}_{\{j\notin I\}}\\
&=\sum_{|I|=i-1} \left(\prod_{i\in I}\lambda_i\right)\sum_{{j\notin I}}\mu_j .
\end{align*}
Since $\lambda_k\geq 0$ for all $1\leq k\leq N$, using
\eqref{II.25 equiv} we obtain \eqref{II.25 reform}.
\end{proof}

The above lemma and Proposition \ref{propII.3} provide the proof of Theorem \ref{thm-mon-1}.

\medskip

Next we proceed to prove Theorem \ref{thm-mon-2} with the help of the regularization procedure
presented in the previous section.

\bigskip

\noindent{\it Proof of Theorem \ref{thm-mon-2}.}
We assume that the functional $\bf F$ is defined
as in \eqref{intro.1} with $f\in C(\sfe)$ and that it is monotonic. 
Then, for every $k\in\N$, let $f_k$ be defined by \eqref{f_k mollification} as in Lemma \ref{l:regularization}
and let ${\bf F}_k$ be the functional given by \eqref{intro.1} with $f$
replaced by $f_k$. Then ${\bf F}_k$ is monotonic as well. Indeed, let $K$ and
$L$  be convex bodies of class $C^2_+$ with support functions $h_K$ and $h_L$,
respectively, and assume that $K\subset L$. Then
\begin{eqnarray*}
&&{\bf F}_k(K)-{\bf F}_k(L)\\
&&=\int_{\sfe} f_k(u)(S_i(Q(h_K,u))-S_i(Q(h_L,u))\,\H^{n-1}(du)\\
&&=\int_{{\bf O}(n)}\omega_k(\rho)\int_{\sfe} f(\rho u)(S_i(Q(h_K,u))-S_i(Q(h_L,u))\,\H^{n-1}(du)\,\nu(d\rho).
\end{eqnarray*}
Now, for each $\rho\in{\bf O}(n)$, we have
\begin{eqnarray*}
&&\int_{\sfe} f(\rho u)(S_i(Q(h_K,u))-S_i(Q(h_L,u))\,\H^{n-1}(du)\\
&&=\int_{\sfe} f(u)(S_i(Q(h_K,\rho^{-1}u))-S_i(Q(h_L,\rho^{-1}u))\,\H^{n-1}(du)\\
&&=\int_{\sfe} f(u)(S_i(Q(h_{\rho K},u))-S_i(Q(h_{\rho L},u))\,\H^{n-1}(du)\\
&&=F(\rho K)-F(\rho L)\le 0,
\end{eqnarray*}
where in the last inequality we have used $\rho K\subset\rho L$ and the monotonicity of ${\bf F}$. 

This proves that
${\bf F_k}$ is monotone for every $k\in\N$. Since $f_k$ is of class $C^2(\sfe)$, it 
satisfies condition $({\bf M})_i$ by Theorem \ref{thm-mon-1}, and this concludes the proof of Theorem \ref{thm-mon-2}.
\begin{flushright}
$\square$
\end{flushright}

\begin{remark}{\rm 
In the introduction, we already pointed out the meaning of condition $({\bf M})_i$ 
 in the special cases $i=n-1$ and $i=1$. Let us consider the case
$i=2$. It can be proved that for every $A$ and
$B$ in $\mat(n-1)$ we have
$$
\tr(S_2^{kj}(A)\cdot B)= \tr(S_2^{kj}(B)\cdot A).
$$
Hence, if $f\in C^2(\sfe)$, condition \eqref{II.25} becomes
$$
\tr(A(S_2^{kj}(Q(f,u))))\ge0\quad\mbox{for every $A\in\mat(n-1)$,
$A>0$,}
$$
for every $u\in\sfe$. This is equivalent to the condition
$(S_2^{kj}(Q(f,u)))\ge0$ for every $u\in\sfe$.}
\end{remark}

\section{Conditions for concavity}\label{sec-BM}

This section is devoted to the proof of Theorem \ref{thm-BM} and some of its extensions.
We consider a functional $\bf F$ of the form
\eqref{intro.1}, and we assume that $\bf F$ is non-negative on
$\K^n$ and  satisfies the Brunn-Minkowski inequality
\begin{equation}\label{III.5}
{\bf F}((1-t)K+tL)^{1/i}\ge (1-t){\bf F}(K)^{1/i}+t{\bf
F}(L)^{1/i}
\end{equation}
for all $ K,L\in\K^n$ and $t\in[0,1]$. 
As noted in the introduction, if $i=1$ then ${\bf F}$ is linear
with respect to Minkowski addition and \eqref{III.5} is satisfied (with equality)
for every $f$. Moreover, the case $i=n-1$ has been settled in
\cite{Colesanti-Hug-Saorin}. Hence we will consider the cases where $2\le i\le n-2$ in the following.

\medskip

\noindent{\em Proof of Theorem \ref{thm-BM}.}
As a first step towards the proof, we show that if ${\bf F}$ is not
identically zero, then ${\bf F}(K)>0$ for every $K\in C^2_+$. Indeed,
as $C^2_+$ bodies are dense in $\K^n$ and $\bf F$ is continuous,
there exists at least one of them, denoted by $K_0$, such that
${\bf F}(K_0)>0$. On the other hand, for any other $K\in C^2_+$, a
suitable rescaled version of $K_0$ is a {\em summand} of $K$, i.e.,
there exists $K'\in\K^n$ and $\lambda\in(0,1)$ such that
$K=(1-\lambda)K'+\lambda K_0$ (see \cite[Corollary 3.2.13]{Schneider}).
From \eqref{III.5} it follows immediately that
${\bf F}(K)\ge\lambda^i{\bf F}(K_0)>0$. On the other hand, 
if ${\bf F}$ is identically zero, then, in particular, it is monotonic so
that condition $({\bf M})_i$ holds (cf. Theorem \ref{thm-mon-1}). From now on we will assume that
${\bf F}$ is strictly positive for $C^2_+$ convex bodies.

\medskip

Consider $K\in\K^n$ of class $C^2_+$ and denote by $h$ its support function, then
$h\in{\mathfrak C}$. For $\phi\in C^\infty(\sfe)$, let $\epsilon>0$ be such that
$$
h_s:=h+s\phi\in\cf\quad\mbox{for every $s$ such that
$|s|\le\epsilon$.}
$$
For $s\in[-\epsilon,\epsilon]$ let $K_s\in{\mathfrak C}$ be such that
$h_s=h_{K_s}$. We compute the first and second derivatives
of $H(s):={\bf F}(K_s)$ at $s=0$. In fact,  we already saw in
\eqref{II.5} that
$$
 H'(s)= \int_{\sfe}f S_i^{kj}(Q(h_s))
q_{kj}(\phi)\,d\H^{n-1}
$$
(recall that we use the convention that we sum over repeated
indices). As $f\in C^2(\sfe)$, applying Lemma \ref{Cheng-Yau} to the last equality we get
\begin{equation*}
 H'(s)= \int_{\sfe}\phi S_i^{kj}(Q(h_s))
q_{kj}(f)\,d\H^{n-1}.
\end{equation*}
Differentiating once more with respect to $s$ (at $s=0$) and using
the notation introduced in Section \ref{subsec-esm}, we obtain
$$
 H''(0)= \int_{\sfe}\phi
S_i^{kj,rs}(Q(h)) q_{kj}(f) q_{rs}(\phi)\,d\H^{n-1}.
$$
Since $K_{(1-\lambda)s+\lambda s'}=(1-\lambda)K_s+\lambda K_{s'}$, 
 inequality \eqref{III.5} yields that the function $s\mapsto
G(s):=H(s)^{1/i}$ is concave for $s\in[-\epsilon,\epsilon]$.
Since ${\bf F} (K)>0$, $H(0)>0$ and $G$ is twice
differentiable at $s=0$, we  conclude that $G''(0)\le 0$, and hence
$$
 H(0)H''(0)-\frac{i-1}{i}H'(0)^2\le0.
$$
This implies
\begin{align}\label{III.15}
&{\bf F}(K)\cdot\ \int_{\sfe}\phi S_i^{kj,rs}(Q(h)) q_{kj}(f)
q_{rs}(\phi)\, d\H^{n-1} \nonumber\\
&\qquad\qquad \le\frac{i-1}{i}\left(\int_{\sfe}\phi
S_i^{kj}(Q(h)) q_{kj}(f)\, d\H^{n-1}\right)^2
\end{align}
for every $h\in\cf$ and $\phi\in C^2(\sfe)$.
For brevity, we set
\begin{equation}\label{III.20-}
M=(m_{rs}(u))_{r,s=1,\dots,n-1}:=
{\bf F}(K)\cdot\left( S_i^{kj,rs}(Q(h,u)) q_{kj}(f,u)
\right)_{r,s=1,\dots,n-1}
\end{equation}
for $ u\in\sfe$. 
Integrating by parts and using Lemma \ref{Cheng-Yau-ext}, we  rewrite \eqref{III.15} in the form
\begin{equation}\label{III.15bis}
\int_{\sfe}\phi^2{\rm trace}(M)\,d\H^{n-1}\le
\int_{\sfe}m_{rs}\phi_r\phi_s\,d\H^{n-1}+
\left(\int_{\sfe}\phi g\,d\H^{n-1}\right)^2,
\end{equation}
for every $\phi\in C^\infty(\sfe)$, where
\begin{equation}\label{III.21}
g(u)=\sqrt{\frac{i-1}{i}}S_i^{kj}(Q(h,u)) q_{kj}(f,u)\,, \quad u\in\sfe.
\end{equation}
The next step, which is the crucial part of the proof,
is to show that \eqref{III.15bis} implies the pointwise matrix condition
\begin{equation}\label{III.20}
\left( S_i^{kj,rs}(Q(h,u)) q_{kj}(f,u)
\right)_{r,s=1,\dots,n-1}\ge0
\end{equation}
for all $ u\in\sfe$. 
For this, we need a result similar to Lemma 3.3 in
\cite{Colesanti-Hug-Saorin}, which is Lemma \ref{technical-lemma}
presented at the end of this proof. This result applied to
\eqref{III.15bis} immediately gives \eqref{III.20}. In particular,
as $Q(h)\ge0$ on $\sfe$, we get
\begin{equation*}
S_i^{kj,rs}(Q(h,u)) q_{kj}(f,u)q_{rs}(h,u)\ge0 
\end{equation*}
for all $u\in\sfe$. 
On the other hand, by the homogeneity of the elementary symmetric
function $S_i$ and its derivatives, we have
\begin{equation*}
S_i^{kj,rs}(Q(h)) q_{rs}(h) =c\cdot S_i^{kj}(Q(h))
\end{equation*}
for some constant $c>0$ and for every $k,j\in\{1,\dots,n-1\}$.
Hence
$$
S_i^{kj}(Q(h,u)) q_{kj}(f,u)\ge0 ,\quad  u\in\sfe,
$$
for every $h\in\cf$,
that is, condition \eqref{II.20}, which is equivalent to the monotonicity of $\bf F$ and
also to condition $({\bf M})_i$ (see Proposition \ref{propII.1} and comments below).
Hence  Theorem \ref{thm-BM} is proved.
\begin{flushright}
$\square$
\end{flushright}

\begin{lema}\label{technical-lemma}
For $r,s\in\{1,\dots,n-1\}$ and $m_{rs}\in C(\sfe)$ let $M:=(m_{rs})_{r,s=1,\dots,n-1}$ and $g\in C(\sfe)$.
If inequality \eqref{III.15bis} holds for every $\phi\in C^\infty(\sfe)$,
then $M(u)\ge0$ for every $u\in\sfe$.
\end{lema}

The proof follows the lines of that of Lemma 3.3 in \cite{Colesanti-Hug-Saorin}; we provide it for the reader's convenience.

\begin{proof}

By standard approximation \eqref{III.15bis} can be extended to every
$\phi\in C(\sfe)$ which is Lipschitz on $\sfe$ (interpreting the first derivatives of $\phi$ as
functions defined $\H^{n-1}$-a.e. on $\sfe$).

\medskip

We proceed by contradiction. Let us assume there exist $\bar u\in\s^{n-1}$ and $\bar v=(\bar v_1,\ldots,\bar v_{n-1})\in \R^{n-1}$ such that
\[
\sum_{r,s=1}^{n-1}m_{rs}(\bar u)\bar v_r\bar v_s < 0\,.
\]
Without loss of generality we may assume $\bar u=(0,\dots,1)$ and $\bar v=(1,\dots,0)$. Then
 we have 
\[
\sum_{r,s=1}^{n-1}m_{rs}(\bar u)\bar v_r\bar v_s=m_{11}(\bar u)<0\,.
\]
We identify $H:=\{x=(x_1,\dots,x_n)\in\R^n\,:\,x_n=0\}$ with $\R^{n-1}$ and, for $\rho\in(0,1)$, we set
\begin{align*}
D_\rho:&=\{(x_1,\dots,x_{n-1})\in\R^{n-1}:\abs{x_i}\le\rho\,,\,i=1,\dots,n-1\}\,,\\
\tilde D_\rho:&=\{u=(u_1,\dots,u_n)\in\sfe\,:\,u_n>0\,,\,(u_1,\dots,u_{n-1})\in D_\rho\}\,.
\end{align*}
We construct a Lipschitz function $\phi$ such that inequality
\eqref{III.15bis} fails to be true. Define first
$\bar g:[-1,1]\rightarrow \R_+ $ as $\bar g(t)=1-\abs{t}$, and
denote by $g(t)$ the periodic extension of $\bar g$ to
the whole real line. Let $\epsilon>0$ and define $g_{\epsilon}(x)=\epsilon
g(x/\epsilon)$. Notice that $g_{\epsilon}\searrow 0$ uniformly on $\R$, as
$\epsilon$ tends to $0$. Let
\[
G(t):=\left\{
\begin{array}{ll} 1, & \text{for } t\in[-1/2,1/2], \\
                          0, & \text{for } |t|\geq 1,\\
                          \text{linear extension}, & \text{otherwise.}
\end{array}\right.
\]
Then $G$ is a bounded Lipschitz function in $\R$. Let us fix
$\rho\in(0,1)$. The function
\begin{eqnarray*}
\Phi_{\epsilon}(x_1,\dots, x_{n-1})=
g_{\epsilon}(x_1)G(x_1/\rho)\dots G(x_{n-1}/\rho)\,,\quad(x_1,\dots,x_n)\in D_\rho,
\end{eqnarray*}
is a bounded Lipschitz function in $D_\rho$ and ${\rm sprt}(\Phi_{\epsilon})\subset D_\rho$.
For $k\ne1$ we have
\[
\frac{\partial\Phi_{\epsilon}}{\partial
x_k}(x_1,\dots,x_{n-1})=\frac{1}{\rho} g_{\epsilon}(x_1)
G'(x_k/\rho)\prod_{j\neq k}{G(x_j/\rho)}\,,
\]
for $\H^{n-1}$-a.e. $(x_1,\dots,x_{n-1})\in D_\rho$. 
As $G\le 1$, $\abs{G'}\le 2$ and $\abs{g_\epsilon}\le\epsilon$ in $\R$,
$$
\abs{\frac{\partial\Phi_\epsilon}{\partial x_k}}
\le\frac{2\epsilon}{\rho}\,\quad\mbox{$\H^{n-1}$-a.e. in $D_\rho$,}
$$
and then
\begin{equation}
\label{IV.0aa}
\lim_{\epsilon\searrow 0}\frac{\partial\Phi_\epsilon}{\partial x_k}=0\,,\quad
\mbox{$\H^{n-1}$-a.e. in $D_\rho$.}
\end{equation}
On the other hand, for $k=1$ and for $\H^{n-1}$-a.e. $(x_1,\dots,x_{n-1})\in D_\rho$
\[
\frac{\partial\Phi_\epsilon}{\partial x_1}(x_1,\dots,x_{n-1})=
\frac{1}{\rho}g_\epsilon(x_1)G'(x_1/\rho)\prod_{j>
1}{G(x_j/\rho)}+g_{\epsilon}'(x_1)\prod_{j=1}^{n-1}{G(x_j/\rho)}\,.
\]
As $\abs{g'_\epsilon}=1$ holds $\H^1$-a.e. in $\R$, we get
\begin{equation}\label{IV.0b}
\abs{\frac{\partial\varPhi_\epsilon}{\partial
x_1}}(x_1,\dots,x_{n-1}) \longrightarrow\prod_{j=1}^{n-1}
G(x_j/\rho)\;\mbox{for $\H^{n-1}$-a.e. $(x_1,\dots,x_{n-1})\in D_\rho$},
\end{equation}
as $\epsilon\searrow0$. In particular, the above limit equals one $\H^{n-1}$-a.e. in $D_{\rho/2}$.
Consider the function
$$
\phi_\epsilon(u)=\phi_{\epsilon}(u_1,\dots,u_n):=\Phi_{\epsilon}(u_1,\dots,u_{n-1})\,,\quad
\,u\in\tilde D_\rho\,,
$$
and extend $\phi_\epsilon$ to be zero in the rest of the unit sphere $\sfe$. In the sequel, for
$u=(u_1,\dots,u_n)\in\tilde D_\rho$, we set $u'=(u_1,\dots,u_{n-1})\in D_\rho$. As $\rho<1$, the
support of $\phi_\epsilon$ is contained in the open hemisphere $\sfe\cap\{x_n>0\}$. We may take $\rho$
small enough such that there exists a local orthonormal frame of coordinates on $\tilde D_\rho$. Taking
covariant derivatives with respect to this frame, by (\ref{III.15bis}) we have
\[
\int_{\s^{n-1}}{\phi_{\epsilon}^2\,{\rm
trace}\left(M\right)\,d\H^{n-1}(u)}\leq
\int_{\s^{n-1}}{\sum_{j,k=1}^{n-1}(\phi_{\epsilon})_j(\phi_{\epsilon})_k\,m_{jk}\,d\H^{n-1}}+
\left(\int_{\sfe}\phi_\epsilon g\,d\H^{n-1}\right)^2.
\]
Since $\Phi_\epsilon$ converges to zero uniformly as $\epsilon\searrow0$, the same is valid for $\phi_\epsilon$,
hence
\begin{equation}
\label{IV.0c}
0\leq
\liminf_{\epsilon\searrow 0}\int_{\s^{n-1}}{\sum_{j,k=1}^{n-1}{m_{jk}(\phi_{\epsilon})_j(\phi_{\epsilon})_k}}\,d\H^{n-1}\,.
\end{equation}
The covariant derivatives of $\phi_{\epsilon}$ can be computed in
terms of partial derivatives of $\Phi$ with respect to Cartesian
coordinates on $D_\rho$; in particular, there exists a
$(n-1)\times(n-1)$ matrix $C=(c_{rs})_{r,s=1,\dots,n-1}$,
depending on $u$, with
$c_{rs}\in\c(D_\rho)$ for $r,s=1,\dots,n-1$, such that
\[
(\phi_{\epsilon})_j(u)=
\sum_{s=1}^{n-1}c_{js}(u')\frac{\partial\Phi_\epsilon}{\partial
x_s}(u')\,,\quad\mbox{for $\H^{n-1}$-a.e. $u\in\tilde D_\rho$.}
\]
We may assume that $C(\bar u')=C(0,\dots,0)$ is the identity matrix. Then, for
$\H^{n-1}$-a.e. $u\in\tilde D_\rho$, 
\[
\sum_{j,k=1}^{n-1}{m_{jk}(u)(\phi_{\epsilon})_j(u)(\phi_{\epsilon})_k(u)}=
\sum_{j,k=1}^{n-1}\sum_{r,s=1}^{n-1}m_{jk}(u)c_{js}(u')c_{kr}(u')\frac{\partial\Phi_\epsilon}{\partial
x_r}(u')\frac{\partial\Phi_\epsilon}{\partial x_s}(u')\,.
\]
This expression is bounded, by the boundedness of the partial derivatives of $\Phi_\epsilon$. Moreover, by
(\ref{IV.0aa}) and (\ref{IV.0b}),
$$
\lim_{\epsilon\searrow 0}
\sum_{j,k=1}^{n-1}{m_{jk}(u)(\phi_{\epsilon})_j(u)(\phi_{\epsilon})_k(u)}=
\prod_{s=1}^{n-1}G^2(u_s/\rho)
\sum_{j,k=1}^{n-1}m_{jk}(u)c_{j1}(u')c_{k1}(u')\,.
$$
Note that
$$
\sum_{i,j=1}^{n-1}m_{ij}(\bar u)c_{i1}(\bar u')c_{j1}(\bar u')=m_{11}(\bar u)<0\,.
$$
Consequently, we may choose $\rho$ sufficiently small so that
$$
\sum_{i,j=1}^{n-1}m_{ij}(u)c_{i1}(u')c_{j1}(u')\le c<0\,,\quad u\in\tilde D_\rho\,.
$$
Then
\begin{eqnarray*}
&&\lim_{\epsilon\searrow0}\int_{\sfe}\sum_{i,j=1}^{n-1}{m_{ij}(\phi_{\epsilon})_i(\phi_{\epsilon})_j}\, d\H^{n-1}\\
&&\qquad =
\int_{\sfe}\prod_{i=1}^{n-1}G^2(u_i/\rho)\sum_{i,j=1}^{n-1}m_{ij}(u)c_{i1}(u')c_{j1}(u')\,\H^{n-1}(du)\\
&&\qquad =\int_{\tilde D_\rho}\prod_{i=1}^{n-1}G^2(u_i/\rho)\sum_{i,j=1}^{n-1}m_{ij}(u)c_{i1}(u')c_{j1}(u')\,\H^{n-1}(du)\\
&&\qquad \le\int_{\tilde D_{\rho/2}}\sum_{i,j=1}^{n-1}m_{ij}(u)c_{i1}(u')c_{j1}(u')\,\H^{n-1}(du)\\
&&\qquad \le c\, \H^{n-1}(\tilde D_{\rho/2})<0\,,
\end{eqnarray*}
which contradicts (\ref{IV.0c}).
\end{proof}

\noindent
{\em Proof of Theorem \ref{thm-BM-sym}.}
Proceeding as in the first part of the proof of Theorem \ref{thm-BM} (but without integration by parts), we arrive at
the following inequality (see \eqref{III.15})
\begin{eqnarray}\label{III.30}
&&{\bf F}(K)\cdot\ \int_{\sfe}f S_i^{kj,rs}(Q(h)) q_{kj}(\phi)
q_{rs}(\phi)\,d\H^{n-1} \nonumber\\
&&\qquad \le\frac{i-1}{i}\left(\int_{\sfe} f
S_i^{kj}(Q(h)) q_{kj}(\phi)\,d\H^{n-1}\right)^2,
\end{eqnarray}
for every $h\in\cf$ and every $\phi\in C^\infty(\sfe)$.  If, in particular, $h$ is even
(and hence is the support function of a centrally symmetric convex body) and $\phi$ is odd, then the right hand-side
of \eqref{III.30} vanishes (as $f$ is even), being the integral of an odd function on $\sfe$.
We may assume, as in the previous proof, that ${\bf F}(K)>0$ for every $K\in C^2_+$. Hence 
we get
\begin{equation}\label{III.30b}
\int_{\sfe}f S_i^{kj,rs}(Q(h)) q_{kj}(\phi)
q_{rs}(\phi)\,d\H^{n-1}\le0
\end{equation}
for every $h\in\cf$ even and $\phi\in C^\infty(\sfe)$ odd. Let $\psi\in C^\infty(\sfe)$ be such that its support is contained in a
open hemisphere $H^+$, and let $\phi\,:\,\sfe\to\R$ be defined by
$$
\bar\phi(u)=
\begin{cases}
\psi(u),& \text{if } u\in H^+,\\
-\psi(-u),& \text{if } u\in\sfe\setminus H^+.
\end{cases}
$$
Then $\bar\phi\in C^\infty(\sfe)$  is an odd function. For this choice of $\bar\phi$ in
\eqref{III.30b}, in view of the symmetry of $f$ and $h$, we get
\begin{equation}\label{III.30c}
\int_{\sfe}f S_i^{kj,rs}(Q(h)) q_{kj}(\psi)
q_{rs}(\psi)\,d\H^{n-1}\le0,
\end{equation}
for every $\psi\in C^\infty(\sfe)$ with support contained in an open hemisphere. We can now apply to $f$
the regularization procedure indicated in Section \ref{sec-preliminaries} in a similar way as in the proof
of Theorem \ref{thm-mon-2}. As the left hand-side of
\eqref{III.30c} is linear with respect to $f$, we obtain that $f_k$ satisfies \eqref{III.30c} as well,
for every $k\in\N$.  Note that the functions on which the proof of Lemma \ref{technical-lemma}
is based are all supported in an open hemisphere. Hence we can apply this lemma to $f_k$
and conclude that it satisfies condition \eqref{II.20} for every $h\in\cf$ even. By Lemma \ref{lemmaII.1}
(note that the proof of this lemma requires the use of even functions only) we have that
condition \eqref{II.25} holds and then, via Lemma \ref{lemma-equiv}, condition $({\bf M})_i$ holds as well.
\begin{flushright}$\square$
\end{flushright}

\bigskip

We conclude this section with the following variant of Theorem \ref{thm-BM} in which the regularity
assumption on $f$ is weakened, and the symmetry hypothesis appearing in \ref{thm-BM-sym} is replaced by the assumption that $f$ belongs to $W^{2,1}(\sfe)$,
the Sobolev space of functions in $L^1(\sfe)$ having second weak derivatives in $L^1(\sfe)$.

\begin{theorem}\label{thm-BM-Sobolev} Let $i\in\{1,\dots,n-1\}$, let $f\in W^{2,1}(\sfe)$ be continuous, and let
${\bf F}$ be defined as in \eqref{intro.1}. If $\bf F$ is non-negative and satisfies inequality \eqref{BM}, 
then there exists a sequence $f_k\in C^2(\sfe)$,
$k\in\N$, which converges uniformly  on $\sfe$ to $f$ such that $f_k$ satisfies condition $({\bf M})_i$ for every $k\in\N$.
In particular, $\bf F$ is monotonic.
\end{theorem}

\begin{proof} As in the proof of Theorem \ref{thm-BM}, the validity of Brunn-Minkowski inequality \eqref{BM} implies \eqref{III.15bis},
where $M$ and $g$, defined by \eqref{III.20-} and \eqref{III.21} respectively, involve weak second derivatives of $f$. 
Moreover, we can assume that $F$ is positive on convex bodies of class $C^2_+$. 
Then inequality \eqref{III.15bis} can be restated in the form
\begin{eqnarray}\label{IV.3}
&&\int_{\s^{n-1}}
\phi(u)^2\,{\rm trace}\left(S_i^{kj,rs}(Q(h,u))q_{kj}(f,u)\right)\, \H^{n-1}(du)\nonumber\\
&&\qquad\qquad\qquad\qquad-
\frac{i-1}{iF(K)}\left(\int_{\s^{n-1}}\phi(u) S_i^{kj}(Q(h,u))q_{kj}(f,u)\, \H^{n-1}(du)\right)^2\nonumber\\
&&\qquad\qquad\le \int_{\s^{n-1}}\sum_{i,j=1}^{n-1}S_i^{kj,rs}(Q(h,u))q_{kj}(f,u)\phi_i(u)\phi_j(u)\, \H^{n-1}(du),
\end{eqnarray}
for all $K\in\mathcal{K}^n$ of class $C^2_+$ with support function $h$ and all $\phi\in C^\infty(\s^{n-1})$. 
For $\rho\in \mathbf{O}(n)$, $\rho K$ is also of class $C^2_+$ and its support function is $h\circ \rho^{-1}$. 
Applying now \eqref{IV.3} with $\rho K$, $h\circ\rho^{-1}$, and $\phi=\phi_\epsilon\circ\rho^{-1}$, where $\phi_\epsilon$ 
is as in the proof of Lemma \ref{technical-lemma}. Multiplying both sides of \eqref{IV.3} with the mollifier $\omega_l$, $l\in\N$, 
 integrating over the rotation group and using the rotation invariance of Hausdorff measures, we get on the right-hand side 
\begin{eqnarray*}
&&\int_{\mathbf{O}(n)}\omega_l(\rho)\int_{\s^{n-1}}\sum_{i,j=1}^{n-1}S_i^{kj,rs}(Q(h\circ\rho^{-1},\rho u))q_{kj}(f,\rho u)\\
&&\qquad\qquad\qquad\qquad\qquad\times ((\phi_\epsilon)\circ\rho^{-1})_i(\rho u)
((\phi_\epsilon)\circ\rho^{-1})_j(\rho u)\, \H^{n-1}(du)\, \nu(d\rho)\\
&&\qquad=\int_{\mathbf{O}(n)}\int_{\s^{n-1}}\sum_{i,j=1}^{n-1}S_i^{kj,rs}(Q(h,v))\omega_l(\rho)q_{kj}(f_\rho,v)(\phi_\epsilon)_i(v)(\phi_\epsilon)_j(v)
 \H^{n-1}(dv)\, \nu(d\rho)\\
&&\qquad=\int_{\s^{n-1}}\sum_{i,j=1}^{n-1}S_i^{kj,rs}(Q(h,v)) q_{kj}(f_l,v)(\phi_\epsilon)_i(v)(\phi_\epsilon)_j(v)
 \H^{n-1}(dv).
\end{eqnarray*}
On the other hand, we can bound the resulting two integrals on the left-hand side by
\begin{align*}
&\left|\int_{\mathbf{O}(n)}\omega_l(\rho)\int_{\s^{n-1}}
((\phi_\epsilon)\circ\rho^{-1})(u))^2\,{\rm trace}\left(S_i^{kj,rs}(Q(h\circ\rho^{-1},u))q_{kj}(f,u)\right)\, \H^{n-1}(du)\,\nu(d\rho)\right|\\
&\qquad\le \int_{\mathbf{O}(n)}\omega_l(\rho)\|\phi_\epsilon\|^2_{L^\infty(\s^{n-1})} c_1(h)\|f\|_{W^{1,2}(\s^{n-1})}\, \nu(d\rho)\\
&\qquad\le c_2(h,f)\|\phi_\epsilon\|^2_{L^\infty(\s^{n-1})}
\end{align*}
and
\begin{align*}
&\left|\int_{\mathbf{O}(n)}\omega_l(\rho)
\frac{i-1}{iF(\rho K)}\left(\int_{\s^{n-1}}(\phi_\epsilon)\circ\rho^{-1})(u) S_i^{kj}(Q(h\circ\rho^{-1},u))q_{kj}(f,u)\, \H^{n-1}(du)\right)^2\, \nu(d\rho)\right|\\
&\qquad\le \int_{\mathbf{O}(n)}\omega_l(\rho)\left(\min\{F(\rho K):\rho\in \mathbf{O}(n)\}\right)^{-1}\|\phi_\epsilon\|^2_{L^\infty(\s^{n-1})}c_3(h)
\|f\|_{W^{1,2}(\s^{n-1})}^2\, \nu(d\rho)\\
&\qquad\le c_4(h,f)\|\phi_\epsilon\|^2_{L^\infty(\s^{n-1})}.
\end{align*}
The constants $c_1,\ldots,c_4$ depend only on the parameters indicated in brackets. Here we use that the minimum $\min\{F(\rho K):\rho\in \mathbf{O}(n)\}$ is  positive and depends only on $f$ and $K$, since $\rho\mapsto F(\rho K)$ is continuous and positive. 
From $\|\phi_\epsilon\|_{L^\infty(\s^{n-1})}\to 0$ as $\epsilon\searrow 0$, we now deduce that
$$
\liminf_{\epsilon\searrow 0}\int_{\s^{n-1}}
\sum_{i,j=1}^{n-1}S_i^{kj,rs}(Q(h,v)) q_{kj}(f_l,v)
(\phi_\epsilon)_i(v)(\phi_\epsilon)_j(v)
\,\H^{n-1}(dv)\ge 0
$$
for all $l\in\N$. Since $f_l\in C^\infty(\s^{n-1})$, we can apply the argument used in the proof of Lemma 
 \ref{technical-lemma} to see that the matrix 
$$
\left(S_i^{kj,rs}(Q(h,v)) q_{kj}(f_l,v)\right)_{r,s=1,\ldots,n-1}
$$
is positive-semidefinite for all $l\in\N$, $h\in{\mathfrak C}$ and all $v\in\s^{n-1}$. 

From this point, we 
can proceed as in the proof of Theorem \ref{thm-BM} after \eqref{III.20}, obtaining that $f_l$ satisfies condition  $({\bf M})_i$.
\end{proof}

\section{Proof of Theorem \ref{thm-general-BM+}} 
This section contains the proof of Theorem \ref{thm-general-BM+}, and hence of Theorem \ref{thm-general-BM}, preceded by an 
auxiliary lemma. 

For the proof we proceed by induction over the dimension $n\ge 3$. The proof uses in an essential way
the special case $i=n-1$ treated in \cite{Colesanti-Hug-Saorin}. 

We start with an auxiliary lemma. For this, let  $\delta_x$ denote the Dirac measure 
with unit mass in the point $x\in\R^n$. We denote by $V^E$ the mixed volume of convex bodies
contained in $E$, defined on $\K^{\dim E}$.

\begin{lema}\label{l: mixed area measure cylinder}
Let $E$ be an $(n-1)$-dimensional linear subspace in $\R^n$ with
unit normal $u$. Let $K_1,\dots,K_{n-2}\subset E$ be convex
bodies. Let $B=B^n\cap E$, where $B^n$ is the Euclidean unit ball
in $\R^n$ and $R\in\R$, $R>0$. If $\eta\subset 
\s^{n-1}$ is an arbitrary Borel set, then
\begin{align*}
&S(K_1,\dots,K_{n-2},B+R[-e_n,e_n];\eta)\\
&\qquad= \frac{R}{n-1} S^E(K_1,\dots,K_{n-2};\eta\cap E)+
V^E(K_1,\dots,K_{n-2},B)(\delta_u+\delta_{-u})(\eta).
\end{align*}
\end{lema}

\begin{proof}
Without loss of generality we can assume that $u=e_n$. Using the
linearity of the surface area measures, we have that
\begin{align*}
& S(K_1,\dots,K_{n-2},B+R[0,e_n];\eta)\\
&\qquad =S(K_1,\dots,K_{n-2},B;\eta)+RS(K_1,\dots,K_{n-2},[0,e_n];\eta).
\end{align*}

Let $L\in\K^n$ be an arbitrary convex body  with support function $h_L$ and
$K_1,\dots,K_{n-2}\in\K^n$. Then we have
\begin{align*}
& \int_{\s^{n-1}} h_L(u)\, S(K_1,\dots,K_{n-2},[0,e_n];du)\\
&\qquad=nV(L,K_1,\dots,K_{n-2},[0,e_n])\\
&\qquad=V^E(L|_{E},K_1|_{E},\dots,K_{n-2}|_{E})\\
&\qquad=\frac{1}{n-1}\int_{\s^{n-1}}{h(L|_{E},u)}\, S^E(K_1,\dots,K_{n-2};du)\\
&\qquad= \frac{1}{n-1}\int_{\s^{n-1}}{h(L,u)\mathbf{1}}_{E}(u)\, S^E(K_1,\dots,K_{n-2};du),
\end{align*}
where we used \cite[(5.68)]{Schneider}.
Hence we obtain that
\[S(K_1,\dots,K_{n-2},[0,e_n];\eta)=\frac{1}{n-1}S^E(K_1,\dots,K_{n-2};\eta\cap E).\]

In order to prove that
$$
S(K_1,\dots,K_{n-2},B;\eta)=V^E(K_1,\dots,K_{n-2},B)\left(\delta_{e_n}+\delta_{-e_n}\right)(\eta)
$$
we observe that for a convex body $K\subset E$, it is known (see \cite[p.~220-221]{Schneider}) that
\[S(K[n-1];\cdot)=(\delta_u(\cdot)+\delta_{-u}(\cdot))V^E(K).\]
Considering $K=\sum_{i=1}^{n-2}\alpha_iK_i+\alpha_{n-1}B$, 
using the multilinearity of area measures and mixed volumes, and then comparing corresponding coefficients of both expressions,  we obtain 
\[S(K_1,\dots,K_{n-2},B;\eta)=V^E(K_1,\dots,K_{n-2},B)(\delta_{e_n}+\delta_{-e_n})(\eta),\]
which finishes the proof.
\end{proof}

\noindent{\em Proof of Theorem \ref{thm-general-BM+}.} 
We proceed by induction on
$n\ge 3$ with $2\le i\le n-1$. The first step of the induction is the case $n=3$, and hence $i=2$. 
More generally, for $n=i+1\ge 3$, 
we know from \cite[Theorem 1.1]{Colesanti-Hug-Saorin} that the assumption implies that 
$f$ is the support function of a convex body. Notice that in this
case  the integration defining the 
functional ${\bf F}$ involves the usual surface
area measure and there are no other convex bodies.

Now we assume that the result is true for all $(n-1)$-dimensional Euclidean subspaces of $\R^n$ and $2\le i\le n-2$. We prove that
 inequality \eqref{BM-min} for the functional
\eqref{F mixed area measure} defined on $\K^n$ and with $ i\in\{2,\ldots, n-1\}$ implies that $f$ is
a support function. 
Since the case $i=n-1$ is already covered by \cite[Theorem 1.1]{Colesanti-Hug-Saorin}, we can assume that $2\le i\le n-2$. 
For this, let $f\in C(\s^{n-1})$ be such that  the functional ${\bf F}$ given in 
\eqref{F mixed area measure} satisfies \eqref{BM-min}, for all $K_1,\dots,K_{n-i-1}\in\K^n$. 

Let $E$ be any
$(n-1)$-dimensional subspace of $\R^n$. Without loss of generality we can
choose $E=\{x\in\R^n\,:\,\langle x,e_n\rangle =0\}=:e_n^{\perp}$ and identify it
with $\R^{n-1}$. Let $B=B^{n}\cap E$ 
and $R\in\R$, $R>0$. 
For $\overline{K},\overline{K}_1,\dots,\overline{K}_{n-i-2}\in\K^{n-1}$
(arbitrary) define $\bar{F}: \K^{n-1}\longrightarrow \R$ by
\[
\overline{{\bf F}}(\overline{K})=\int_{\s^{n-1}}f(x)\,S(\overline{K}[i],
\overline{K}_1,\dots,\overline{K}_{n-i-2},B+R[0,e_n];dx).
\]
We notice that as $2\le i\le n-2$, we have $n\ge 4$. 

From the assumption  \eqref{BM-min} on ${\bf F}$, it follows that $\overline{{\bf
F}}$ satisfies
\begin{equation}\label{WBM for Fbar}
\overline{{\bf F}}((1-t)\overline{K}+t\overline{L})\geq
\min\left\{\overline{{\bf F}}(\overline{K}),\overline{{\bf F}}(\overline{L})\right\} 
\end{equation}
for all $t\in[0,1]$, $\overline{K},\overline{L}\in\K^{n-1}$, and any choice of $\overline{K}_1,\dots,\overline{K}_{n-i-2}\in\K^{n-1}$. 
Lemma \ref{l: mixed area measure cylinder} shows that
\begin{align*}
\overline{{\bf F}}(\overline{K})&=  \frac{R}{n-1}\int_{\s^{n-2}}{f|_{E}(x)}\,S^E(\overline{K}[i],\overline{K}_1,\dots,\overline{K}_{n-i-2};dx) \\
 &\qquad +(f(e_n)+f(-e_n))V^E(\overline{K}[i],\overline{K}_1,\dots,\overline{K}_{n-i-2},B),
\end{align*}
and similarly for $\overline{L}$ and $(1-t)\overline{K}+t\overline{L}$. 
We plug this into \eqref{WBM for Fbar} and divide the resulting inequality by $R$. Then,  for all $\overline{K},\overline{L}\in \K^{n-1}$ and
$t\in[0,1]$, we get
\begin{align*}
&\frac{1}{R}\overline{{\bf F}}((1-t)\overline{K}+t\overline{L})\\
&\qquad=\frac{1}{n-1}\int_{\s^{n-2}}{f|_{E}(x)}\,S^E(((1-t)\overline{K}
+t\overline{L})[i],\overline{K}_1,\dots,\overline{K}_{n-i-2};dx)\\
&\qquad\qquad +\frac{1}{R}(f(e_n)+f(-e_n))V^E(((1-t)\overline{K}+t\overline{L})[i],\overline{K}_1,\dots,\overline{K}_{n-i-2},B)\\
&\qquad\geq \min\left\{\frac{1}{n-1}\int_{\s^{n-2}}{f|_{E}(x)}\,S^E(\overline{K}[i],\overline{K}_1,\dots,\overline{K}_{n-i-2};dx)\right.\\
&\qquad\qquad\qquad\qquad+\frac{1}{R}(f(e_n)+f(-e_n))V^E(\overline{K}[i],\overline{K}_1,\dots,\overline{K}_{n-i-2},B),\\
&\qquad\qquad\qquad \frac{1}{n-1}\int_{\s^{n-2}}{f|_{E}(x)}\,S^E(\overline{L}[i],\overline{K}_1,\dots,\overline{K}_{n-i-2};dx) \\
&\qquad\qquad\qquad\qquad +\left.\frac{1}{R}(f(e_n)+f(-e_n))V^E(\overline{L}[i],\overline{K}_1,\dots,\overline{K}_{n-i-2},B)\right\}.
\end{align*}

When $R$ tends to infinity, we obtain that the functional defined on $\K^{n-1}$ and given by
\[
\overline{K}\mapsto
\int_{\s^{n-2}}{f|_{E}(x)\,S^E(\overline{K}[i],\overline{K}_1,\dots,\overline{K}_{n-i-2};dx)}
\]
satisfies \eqref{BM-min}. Hence, the induction
hypothesis yields that $f|_{E}$ is a convex function in $E$. Since the
same argument works for an arbitrary subspace $E$, we conclude that
$f$ is a convex function, that is, (the homogeneous extension of) $f$ is the support function of a
convex body.
\begin{flushright}
$\square$
\end{flushright}

\end{document}